\documentclass{amsart}
\usepackage{amsmath,amsthm,amsfonts,amssymb,amscd,mathrsfs, stmaryrd}
\usepackage[left=3.5cm,right=3.5cm]{geometry}
\usepackage{graphicx}
\usepackage{color}

\usepackage{bbm}
\usepackage{comment}
\usepackage{tikz}
\usepackage{siunitx}
\usepackage{subfigure}

\usepackage{psfrag}

\usepackage{epic,eepic}
\usepackage{ebezier}

\usepackage{mathtools}      
\usepackage{bm}             
\usepackage{indentfirst}    
\usepackage{mathdots}
\usepackage{thmtools}
\usepackage{enumitem}       
\usepackage[colorlinks=true,backref=page]{hyperref} 
\usepackage[font=small]{caption}
\usepackage{tikz-cd}        
\usetikzlibrary{decorations.pathreplacing}
\usepackage{tikz}
\usetikzlibrary{datavisualization.formats.functions}

\usetikzlibrary{arrows,shapes,positioning}
\usetikzlibrary{decorations.markings}
\tikzstyle arrowstyle=[scale=1]
\tikzstyle directed=[postaction={decorate,decoration={markings,mark=at position .5 with {\arrow[arrowstyle]{stealth}}}}]
\tikzstyle reverse directed=[postaction={decorate,decoration={markings,mark=at position .5 with {\arrowreversed[arrowstyle]{stealth};}}}]
\usetikzlibrary{calc}

\theoremstyle{plain}
\newtheorem{theorem}{Theorem}
\newtheorem*{Mtheorem}{Main Theorem}

\newtheorem{prop}{Proposition}
\newtheorem{lemma}{Lemma}
\newtheorem{cor}{Corollary}
\newtheorem{remark}{Remark}

\newcommand{\N}{\ensuremath{\mathbb{N}}}
\newcommand{\R}{\ensuremath{\mathbb{R}}}

\long\def\begcom#1\endcom{}

\newcommand{\Id}{\operatorname{Id}}

\newcommand{\length}{\operatorname{\length}}

\newcommand{\Jac}{\operatorname{Jac}}

\newcommand{\vep}{\varepsilon}

\def\length{\operatorname{length}}

\def\dim{\operatorname{dim}}

\def\Jac{\operatorname{Jac}}

\def\vep{\varepsilon}

\begin{document}

\title{Symbolic extensions for  3-dimensional  diffeomorphisms}

\author[David Burguet  and Gang Liao]{David Burguet $^{*}$ and Gang Liao $^{**}$}

\email{david.burguet@upmc.fr}

\email{lg@suda.edu.cn}

\date{November, 2019}



\thanks{$^*$LPSM, Sorbonne Universite, Paris 75005, France.  $^{**}$School of Mathematical Sciences, Center for Dynamical Systems and Differential Equations, Soochow University,
	Suzhou 215006,  China; G. Liao  was partially supported by NSFC (11701402,  11790274),  BK 20170327 and IEPJ}

\maketitle


\begin{abstract}We prove that every  $\mathcal{C}^{r}$  diffeomorphism with $r>1$ on a three-dimensional  manifold  admits symbolic extensions, i.e.  topological extensions which are subshifts over a finite alphabet. This answers positively a conjecture of  Downarowicz  and  Newhouse in dimension three.
\end{abstract}

\allowdisplaybreaks

\section{Introduction}\label{sec:intro}

A symbolic extension of a topological dynamical  system  is a topological extension given by a subshift over a finite alphabet. Existence and entropy of  symbolic extensions have been intensively  investigated in the last decades. M. Boyle and T. Downarowicz \cite{BD} characterized the problem of existence in terms of new entropic invariants related to weak expansiveness properties of the system. In particular asymptotically $h$-expansive systems always admit  \textit{principal} symbolic extensions, i.e. extensions that  preserve the entropy of invariant measures \cite{bff}.

For smooth systems on compact manifolds this theory appears to be  of highly interest. It is well known that Markov partitions allow to encode uniformly hyperbolic systems by finite-to-one symbolic extensions of finite type.  Beyond uniform hyperbolicity,   partially hyperbolic diffeomorphisms with one dimensional center satisfy the $h$-expansiveness property, hence admit principal symbolic extensions \cite{CY, LFPV}.   More recently the second author with  M. Viana and J. Yang showed that  smooth systems with no principal symbolic extension are   $\mathcal{C}^1$-close to diffeomorphisms with homoclinic tangencies \cite{LVY}.

Moreover the  existence of symbolic extensions depends on the order of smoothness. While $\mathcal{C}^\infty$ systems are asymptotically $h$-expansive \cite{Buz, Yom}  and thus admit principal symbolic extensions, there is a  $\mathcal{C}^1$ open set of  3-dimensional diffeomorphisms  \cite{asa} (resp. Lebesgue preserving  diffeomorphisms  \cite{BCF, DN})   in which generic  ones   have  no symbolic extension.
In intermediate smoothness, i.e. for $\mathcal{C}^r$ systems with $1<r<+\infty$, the  existence  was conjectured by  T. Downarowicz  and  S. Newhouse in  \cite{DN} and in general this problem is still open.  It has been first  proved for circle maps by T. Downarowicz and A. Maass \cite{DM} and then by the first author for surface diffeomorphisms \cite{Burguet11,Burguet12}. In this paper we work further on  \cite{Burguet12} to show existence of symbolic extensions for diffeomorphisms in dimension 3. We refer to the next section for the definitions and notations used in our main Theorem below.

\begin{Mtheorem}\label{maintheorem}
	Let $f$ be a $\mathcal{C}^{r}$  diffeomorphism with $r>1$ on a compact 3-dimensional manifold $M$. Then $f$ admits a symbolic extension $\pi:(Y,S)\rightarrow (M,f)$ satisfying for  all $\mu\in \mathcal{M}_{inv}(f)$:
	
$$ \max_{\xi\in \mathcal{M}_{inv}(S), \ \pi \xi=\mu}h(S,\xi)=h(f,\mu)+\frac{\lambda_1^+(f,\mu)+\lambda_2^+(f,\mu)}{r-1},$$

\noindent where  $\lambda_1^+(f,\mu)\geq \lambda_2^+(f,\mu)$ denote the positive parts of the two largest Lyapunov exponents of $\mu$.
\end{Mtheorem}

The ingredient of the present advance is mainly a new inequality relating the Newhouse local entropy of an ergodic measure and the local volume growth of smooth discs of unstable  dimension (which is   the number of positive Lyapunov exponents of the measure).
Section 3 is devoted to the proof of this key estimate. Then for a $3$-dimensional diffeomorphism, we may bound from above  the Newhouse local entropy with respect to either $f$ or $f^{-1}$ by the local volume growth of curves, which implies the existence of symbolic extensions  by combining with the \textit{Reparametrization Lemma} developed in  \cite{Burguet12}. This is proved together with the Main Theorem in the last section.

\section{Preliminaries}\label{sec:pre}

\subsection{ Newhouse entropy structure and the Symbolic Extension Theorem.}

Consider a topological system $(M,f)$, i.e. a continuous map  $f:M\rightarrow M$ on a compact metric space $(M,d)$.   For $x\in M, \vep>0, n\in \mathbb{N}$, we  denote the $n$-step dynamical ball  at $x$ with radius $\varepsilon$ by $$B_n(x,\vep,f)=\{y\in M: d(f^i(x), f^i(y))<\vep,\,\,i=0,\cdots,n-1\}.$$ A subset $N$ of $M$ is said  $(n,\delta)$-separated  when  any pair $y\neq z$ in $N$ satisfies $d(f^i(y),f^i(z))>\delta$ for some $i\in [0,n-1]$. For any subset $\Lambda$ of  $M$ and $\delta>0$, denote by $s(n,\delta, \Lambda)$  the maximal cardinality of the $(n,\delta)$-separated sets   contained in $\Lambda$. For any $\Lambda\subset M$, $\vep>0$,   define $$h^*(f,\Lambda, \vep)=\lim_{\delta\to 0}\limsup_{n\to \infty}\frac1n\log \sup_{x\in \Lambda}s\left(n,\delta,B_{n}(x,\vep,f)\cap\Lambda\right).$$
 Denote by $\mathcal{M}_{inv}(f)$ (resp. $\mathcal{M}_{erg}(f))$ the set of all  $f$-invariant (resp. ergodic $f$-invariant) Borel probability measures endowed with the usual metrizable  weak-$*$ topology. Given $\mu\in \mathcal{M}_{erg}(f)$,  for any $\vep>0$, Newhouse \cite{New} defined the tail entropy of $\mu$ at the scale $\vep$ by letting
$$h^*(f, \mu, \vep)=\lim_{\eta\to 1,\,0<\eta<1} \inf_{\mu(\Lambda)>\eta} h^*(f,\Lambda, \vep).$$
For $\mu\in \mathcal{M}_{inv}(f)$, assuming
$\mu=\int_{\mathcal{M}_{erg}(f)}\nu \, dM_\mu(\nu)$ is the ergodic
decomposition of $\mu$, let $$h^*(f, \mu, \vep)=\int_{\mathcal{M}_{erg}(f)} h^*(f,\nu, \vep)\,dM_\mu(\nu).$$

Entropy structures are \textit{particular} non-increasing sequences of nonnegative functions defined on $\mathcal{M}_{inv}(f)$ which are converging pointwisely to the Kolmogorov-Sinai entropy function $h:\mathcal{M}_{inv}(f)\rightarrow \mathbb{R}^+$ (see \cite{Dow} for a precise definition). They satisfy the following criterion for the existence of symbolic extensions.

\begin{theorem}[Symbolic Extension Theorem \cite{BD, DM}]\label{SEX}
Let $(M,f)$ be a topological system. Assume $E$ is a nonnegative affine upper semicontinuous function such that for all $\mu\in \mathcal{M}_{inv}(f)$ there is an entropy structure $(h_k)_k$ satisfying
\begin{equation}\label{fon}\lim_k \limsup_{\mathcal{M}_{erg}(f)\ni\nu\rightarrow \mu}(E+h-h_k)(\nu)\leq E(\mu).\end{equation}
Then there exists a symbolic extension $\pi:(Y,S)\rightarrow (M,f)$ such that
$$ \max_{\xi \in \mathcal{M}_{inv}(S), \ \pi \xi=\mu}h(S,\xi)=(E+h)(\mu).$$
\end{theorem}

Letting $\vep_k\to 0$, then  the sequence $(h^{New}_k)_k$ defined by $h^{New}_k(f, \mu):=h(f, \mu)-h^*(f, \mu, \vep_k)$  for all $k\in \mathbb{N}$ and for all $\mu\in \mathcal{M}_{inv}(f)$  is an entropy structure \cite{Dow}.  As a matter of fact,  for any $m\in \mathbb{Z}\setminus\{0\}$,  $h^{New}_{m, k}(f, \mu):=h(f, \mu)-\frac{1}{|m|}h^*(f^m, \mu, \vep_k)$  for all $k\in \mathbb{N}$ and for all $\mu\in \mathcal{M}_{inv}(f)$  is also an entropy structure (see Lemma 1 in \cite{Burguet12}).

\subsection{Lyapunov exponents}
Let  $f:M\rightarrow M$ be a  differentiable map on a compact  Riemannian manifold $(M,\|\cdot\|)$ of dimension $\mathsf d$.
Given   $x\in M$,  the Lyapunov exponent relative to a direction $v\in T_xM$ is  the exponential growth rate  given  by the limit
\begin{eqnarray}\label{limit}\lim_{n\to \infty}\frac{1}{n}\log \|D_xf^nv\|,\end{eqnarray}
which exists  for almost every point $x$ with respect to every  $f$-invariant Borel  probability measure $\mu$ by Oseledets theorem \cite{Oseledets} (it does not depend on the Riemannian structure on $M$).   Moreover, for $\mu$-almost every point $x$, there exist  values  $\lambda_1(f,x)\ge \cdots\ge\lambda_{\mathsf d}(f,x)$ of the limit (\ref{limit}) and  measurable  flags of the tangent spaces  $\{0\}=G^{\mathsf d+1}_x\subset  G^{\mathsf d}_x\subset \cdots \subset G^{1}_x=T_xM$  satisfying:
$$\lim_{n\to \infty}\frac{1}{n}\log \|D_xf^nv\|=\lambda_i(f,x),\quad \forall v\in G^i_x\setminus G^{i+1}_x,\,\,1\le i\le \mathsf d. $$
 For any $\mu\in  \mathcal{M}_{inv}(f) $, $1\le i\le \mathsf d$,  we denote
\begin{eqnarray*}\lambda_i(f, \mu)=\int  \lambda_i(f,x)\, d\mu(x),\\
\sum_{j=1}^i \lambda^+_j(f, \mu)=\int \sum_{j=1}^i \lambda_j^+(f,x)\, d\mu(x). \end{eqnarray*}
For $\nu\in \mathcal{M}_{erg}(f)$,  we have $\lambda_i(f, \nu)=  \lambda_i(f,x)$ for all $i$ and for  $\nu$-almost every $x$.
By standard arguments the function $\mu\mapsto \sum_{j=1}^i \lambda^+_j(f, \mu)$ defines an affine upper semicontinuous function on $ \mathcal{M}_{inv}(f)$ (see Lemma 3 in \cite{Burguet12}).
For a $\mathcal{C}^{r}$  diffeomorphism  with $r>1$ on a compact 3-dimensional Riemannian manifold,  we will prove that $E=\frac{\sum_{j=1}^2 \lambda^+_j(f, \cdot)}{r-1}$ satisfies Inequality (\ref{fon}), which together with Theorem \ref{SEX} implies the Main Theorem.

\subsection{Nonuniformly hyperbolic estimates}
Assume now $f$ is a diffeomorphism. In this case, Oseledets theorem provides
 for any $\mu\in \mathcal{M}_{inv}(f)$, for $\mu$-a.e., $x\in M$,  a
decomposition on the tangent space
$
T_xM=E^{cs}_x\oplus E^{u}_x
$ and $ \rho_{cs}(x)\le 0<\rho_{u}(x)$
satisfying\\

\begin{itemize}
	\item 
	$\underset{|n|\rightarrow
		\infty}{\lim}\frac{1}{n}\log\|D_xf^n(v)\|\leq \rho_{cs}(x),\quad \forall\,0\neq v\in E^{cs}_x ;$\\[2mm]
	
	\item 
	$\underset{|n|\rightarrow
		\infty}{\lim}\frac{1}{n}\log\|D_xf^n(w)\|\geq \rho_{u}(x),\quad  \forall\,0\neq w\in E^{u}_x ;$
	\\[2mm]
	
	\item 
	$\underset{|n|\rightarrow
		\infty}{\lim}\frac{1}{n}\log\sin\angle(E_{f^n(x)}^{cs},E_{f^n(x)}^{u})=0$.
\end{itemize}
 For $ 0<\gamma\ll \lambda_{u}$ and $k\in \mathbb{N}$,  we  consider the  sets $\Lambda_k(\lambda_{u}, \gamma)$  consisting of points $x$ in $M$ with the following properties:\\

\begin{itemize} \item
$~ \|Df^n|E_{f^i(x)}^{cs}\|\leq  e^{k\gamma
	}e^{|i|\gamma}e^{n\gamma
	}\,, \quad\forall\, i\in\mathbb{Z}, \,n\geq1;$\\[2mm]

	\item
	$~
	\|Df^{-n}|E_{f^i(x)}^{u}\|\leq  e^{k\gamma }e^{|i|\gamma}e^{n(-\lambda_{u}+\gamma)}\,, \quad\forall\, i\in\mathbb{Z}, \,n\geq1;$\\[2mm]
	
	\item
	$~
	\sin\angle(E_{f^i(x)}^{cs},E_{f^i(x)}^{u})\geq e^{-k\gamma
	}e^{-|i|\gamma}\,,\quad \forall\, i\in\mathbb{Z}.$
\end{itemize}
$\,$ \\
From the definition, it holds that  \cite{BP, Pollicott} \\

\begin{itemize} \item~ $T_{x}M=E^{cs}_x\oplus E^u_x$ is a continuous splitting  on each $\Lambda_k(\lambda_u,\gamma)$; \\[2mm]
	
	\item~
	$f^{\pm}(\Lambda_k(\lambda_u,\gamma))\subset \Lambda_{k+1}(\lambda_u,\gamma)$ for any $k\in \mathbb{N}$;\\[2mm]
	
	\item~
$x\in \bigcup_{k\in \mathbb{N}}\Lambda_k( \lambda_{u},  \gamma)$ provided $\lambda_{u}\leq \rho_{u}(x)$;\\[2mm]

	\item~ $\lim_{\lambda_u\to 0}\mu\left(\bigcup_{k\in \mathbb{N}} \Lambda_k(\lambda_{u},  \gamma)\right)$ $=1$ for any $\mu\in \mathcal{M}_{inv}(f)$.

\end{itemize}
  For the sake of statements,  we let $\Lambda_k=\Lambda_k(\lambda_{u},  \gamma)$ for any $k\in \mathbb{N}$ and $\Lambda^*=\bigcup_{k\in \mathbb{N}}\Lambda_k( \lambda_{u},  \gamma)$.
Denote  $\lambda'_{u}=\lambda_{u}-2\gamma$. Given $x\in \Lambda^*$,   define  for all $v=v_{cs}+v_u$ and $w=w_{cs}+w_u$ with  $v_{cs},w_{cs}\in E^{cs}_x$ and $v_{u},w_u\in E^{u}_x$,
\begin{eqnarray*}<v_{cs}, w_{cs}>'&=&\sum_{n=0}^{+\infty}e^{-4n\gamma}<D_xf^n(v_{cs}),D_xf^n(w_{cs})>,
 \\
	<v_{u}, w_{u}>'&=& \sum_{n=0}^{+\infty}e^{2n\lambda'_{u}}<D_xf^{-n}(v_{u}),D_xf^{-n}(w_{u})> ,
	\\
	<v, w>'&=&<v_{u}, w_{u}>'+<v_{cs}, w_{cs}>'.\end{eqnarray*}
There exists $a_1=a_1(\gamma)>1$ such that
\begin{eqnarray} \label{metric}\|v\|\leq \|v\|'\leq a_1e^{ k\gamma}\|v\|,\quad \forall~v\in T_{\Lambda_k}M.\end{eqnarray}The norm $\|\cdot\|'$ is called a Lyapunov metric,  with which $f$ behaves  uniformly on $\Lambda^*$:
\begin{eqnarray*}
&&\frac{1}{\|Df^{-1}\|}\|v_{cs}\|'_x \le	\|D_xf(v_{cs})\|'_{f(x)}\leq e^{2\gamma}\|v_{cs}\|'_x,\\[2mm]
 && \frac{1}{\|Df\|}\|v_{u}\|'_x \le \|D_xf^{-1}(v_{u})\|'_{f^{-1}(x)}\leq\,e^{-\lambda'_{u}}\|v_{u}\|'_x.
\end{eqnarray*}
In this manner,  the splitting $T_{\Lambda^*} M=E^{cs}\oplus E^u$ is  dominated with respect to  $\|\,\|'$, i.e.   $$\frac{\|D_xf(v_{cs})\|'}{\|D_xf(v_u)\|'}\le e^{2\gamma-\lambda_u'}\frac{\|v_{cs}\|'}{\|v_u\|'},\quad \forall\,0\neq v_{cs}\in E^{cs}_x,\,0\neq v_u\in E^u_x,\,\,x\in \Lambda^*,$$
$$\text{with }2\gamma-\lambda_u'<0.$$

We consider  a $\mathcal{C}^r$ diffeomorphism  $f$ on a $C^r$ smooth Riemanian manifold $(M,\| \cdot \|))$ with $r>1$. Let $\alpha=\min\{r-1,1\}$.
We are going to  state that the dominated behavior  on each  $\Lambda_k$ can be extended  to a $e^{-k\mathsf{d}\gamma'}$-neighborhood  for $\gamma'=\alpha^{-1}\gamma$.  Moreover, for attaining a preassigned local proximity  of dominated splitting,  we may choose a positive number $b$ independently of $k$ such that this  proximity holds  in a $be^{-k\mathsf{d}\gamma'}$-neighborhood of $\Lambda_k$.  

Let $d$ be the Riemannian distance on $M$ and $\mathsf{r}$ be the radius of injectivity of $(M,\|\cdot\|)$. The ball at $x\in M$ of radius $R\in \mathbb{R}^+$ with respect to $d$ is denoted by $B(x,R)$.  Then for $y\in B(x,\mathsf{r})$ we use the identification 
\begin{eqnarray*}
T_{B(x,\mathsf{r})}M&\simeq & B(x,\mathsf{r})\times T_xM,\\
(y,v) &\mapsto & \left(y, D_y(\exp_x^{-1})(v)\right) 
\end{eqnarray*}
to  ``translate"   the vector $v\in T_yM$ to the  vector $\hat{v}_x:= D_y(\exp_x^{-1})(v)\in
T_xM$. 
Recall that the exponential map $(x,v)\mapsto \exp_x(v)$ defines a  $C^r$ map (thus $C^{1+\alpha}$) from $TM$ to $M$  with $D_x(\exp_x)=\Id_{T_xM}$.  Since  the  diffeomorphism $f$ is also $C^{1+\alpha}$ on the compact manifold $M$,    there 
exist $K>1,a_2>0$  such that
\begin{eqnarray*}
\forall x\in M, \ \forall (y,v)\in  T_{B(x,a_2)} M, \ & \frac{\|v\|}{2}\leq \|\hat{v}_x\|\leq 2\|v\|\\ 
& \text{ and } \|D_xf^{\pm}(\hat{v}_x)-\widehat{D_yf^{\pm}v}_{f(x)}\|\leq\,K\|v\|d(x,y)^\alpha.
\end{eqnarray*}

 For $x\in \Lambda^*$ and $(y,v)\in  T_{B(x,a_2)} M$, we define  $\|v\|_x''=\|\hat{v}_x\|'$ and we also let $<,>''_x$ be the associated scalar product on $T_yM$.    
 It follows then from (\ref{metric})  that 
\begin{eqnarray}\label{mettt}\forall x\in \Lambda_k, \ \forall (y,v)\in  T_{B(x,a_2)} M, \ \ 2a_1e^{ k\gamma}\|v\|&\geq \|v\|_x''= \|\hat{v}_x\|'&\geq \frac{\|v\|}{2}.\end{eqnarray}
 We write $\hat{v}_x$ as $v$
whenever there is no confusion and we also denote  by $T_yM=E^{cs}_x\oplus E^{u}_x$ the splitting of $T_yM$ which translates to the splitting $T_xM=E^{cs}_x\oplus E^{u}_x$.
Let $\lambda''_{u}=\lambda'_{u}-\gamma$ and let $a'_2>0$ such that $f^{i}(B(x,a'_2))\subset B(f^{i}x,a_2)$ for all $x\in M$ and $i=0,1,-1$. Then
 define
$$\gamma_k=\min\left\{ 1, a'_2, \left(\frac{e^{-\lambda''_u}-e^{-\lambda'_u}}
{4a_1e^{(k+1)\gamma}K}\right)^{\frac{1}{\alpha}}\right\}.$$
Then we have for all $x\in \Lambda_k$, for all $y\in B(x,\gamma_k)$ and for all $v_{cs/u}\in E^{cs/u}_x\subset T_yM$  (see \cite{Pollicott}  p.72 for further details) : 
\begin{eqnarray}\label{s-con} \Big{(}\frac{1}{\|Df^{-1}\|}-(e^{\gamma}-1)\Big{)} \|v_{cs}\|''_x\le \|D_yf(v_{cs})\|''_{f(x)}\leq e^{3\gamma}\|v_{cs}\|''_x,\\
\label{s-con2} \Big{(} \frac{1}{\|Df\|}-(e^{\gamma}-1)\Big{)} \|v_u\|''_x \le\|D_yf^{-1}(v_u)\|''_{f^{-1}(x)}\leq
e^{-\lambda''_{u}}\|v_u\|''_x.\end{eqnarray}

  Define $\kappa(x)=\min\{k\in
\mathbb{N}:  x\in \Lambda_k\}$ for $x\in \Lambda^*$.  Then the inequalities (\ref{s-con}) and (\ref{s-con2}) hold  for any $y\in
B(x,\gamma_{\kappa(x)})$.  Such sets  $B(x,\gamma_{\kappa(x)})$ are called Lyapunov neighborhoods.  Letting $\gamma'=\alpha^{-1}\gamma$, we have  $\gamma_k=a_3e^{-k\gamma'}<1$  for $k$ large enough and some constant $a_3$ independent of $k$.  We use $d''_x$ to denote the distance induced by $\|\cdot\|_x''$ on $B(x,a_2)$ and  $B''_x(y,r)$ to denote the ball centered at $y$ with radius $r$ in $d''_x$.

For the purpose of our use in the  computation of tail  entropy and local volume growth, we need to estimate the proximity of the dominated splitting  in Lyapunov neighborhoods along orbits.  For a splitting $F=F_1\oplus F_2$ of an Euclidean space $F$ with norm $\|\,\|$, and $\xi>0$,  we denote by $Q_{\|\,\|}(F_1,\xi)$ the cone of width $\xi$ of $F_1$ in $\|\,\|$, i.e.  the set $\{v=v_1+v_2\in F:\,v_{1}\in F_1,\,v_2\in F_2,\quad \|v_{2}\|\le \xi \|v_{1}\|\}$. For any vector subspace $G$ of $F$ we let $\iota(G)$ be the Pl\"ucker embedding of $G$ in the projective space $\mathbb{P}\varcurlywedge F$ of the Euclidean power exterior algebra $\varcurlywedge F$. When $A:F\rightarrow F'$ is a linear map  between two  finite dimensional Euclidean spaces $F$ and $F'$, we let $\varcurlywedge^{l}A$ be the induced map on the $l$-exterior power $\varcurlywedge^l F$ with $l$ less than or equal to the dimension of $F$. 
With the above notations the map $x\mapsto\varcurlywedge^l D_x f$ is $\alpha$-H\"older and one may assume its H\"older norm is less than $K$ by taking $K$  larger in advance.  Observe that $ \varcurlywedge^l F\ni u \mapsto \|\varcurlywedge^{l}Au \|$  induces a map on $\mathbb{P}\varcurlywedge^l F$ by letting  $\|\varcurlywedge^{l}A(\mathbb{P}u) \|=\frac{\|\varcurlywedge^{l}Au \|}{\|u\|}$.  Also we let $l_u(f, z)$ be the dimension of $E^u(z)$.  When $\mu\in \mathcal M_{erg}(f)$, $l_u(f, z)$ is a constant for $\mu$-a.e. $z$, which we denote by $l_u(f, \mu)$.  

\begin{lemma}\label{uniform size}
For any $\xi>0$ small enough there exists $a_\xi>0$ such that for any $x\in \Lambda^*$ and for  any $y\in B\left(x,a_\xi\gamma_{\kappa(x)}^{l}\right)$   with $l=l_u(f, x)$ we have :

\begin{itemize} \label{neighborhood}

		\item[(i)] 	$\|D_yf(v)\|''_{f(x)}\geq e^{\lambda''_u-\gamma}\|v\|''_x$ for all $v\in Q_{\|\|_x''}(E^u_x, \xi)$ and  $\|D_yf(v)\|''_{f(x)}\leq e^{4\gamma}\|v\|''_x$ for all $v\in Q_{\|\|_x''}(E^{cs}_x, \xi)$,  \\[2mm]
		
			\item[(ii)]  $D_yf(Q_{\|\|_{x}''}(E^u_x, \xi))\subset Q_{\|\|_{f(x)}''}(E^u_{f(x)},\xi)$ and $D_yf^{-1}(Q_{\|\|_x''}(E^{cs}_x, \xi))\subset Q_{\|\|''_{f^{-1}(x)}}(E^{cs}_{f^{-1}(x)},\xi)$, \\[2mm]

	\item[(iii)] $e^{-\gamma }\leq \frac{\|\varcurlywedge^{l}D_yf(\iota(G))\|_{f(x)}''}{\|\varcurlywedge^{l}D_xf(\iota(E^u_{x}))\|_{f(x)}''}\leq e^{\gamma }$ for all $l$-plane $G\subset Q_{\|\|_x''}(E^u_x,\xi)$.
\end{itemize}

\end{lemma}

\begin{proof}Let $\xi>0$ and  $x\in \Lambda^*$.\\
\begin{itemize}

 \item[(i)] By the domination property  $E^{cs}_x\oplus E^u_x $ at $y$ with respect to $\|\cdot \|_x''$  given by the inequalities  (\ref{s-con}) and (\ref{s-con2}),  the first item holds for  small $\xi$  independent of $\kappa(x)$.\\

\item[(ii)] Using the invariance of $E^u$ and the domination property at $x$ there exists $\varsigma\in (0,1)$ independent of $x$ satisfying $D_xf(Q_{\|\|_x''}(E^u_x, \xi))\subset Q_{\|\|_{f(x)}''}(E^u_{f(x)},\varsigma\xi)$. Then for any $y\in B(x,a_\xi\gamma_{\kappa(x)})$, we get by the Inequalities (\ref{mettt})
\begin{eqnarray*}
\|D_xf-D_yf\|_x''&:= &\max_{\|v\|_x''=1}\|D_xf(v)-D_yf(v)\|_{f(x)}''\\ [2mm]
&\leq& 4a_1e^{\kappa\left(f(x)\right)\gamma}\|D_xf-D_yf\| \\ [2mm]
& \leq& 4Ka_1e^{(\kappa(x)+1)\gamma}(a_\xi\gamma_{\kappa(x)})^{\alpha}\\ [2mm]
&\leq & a_\xi^{\alpha}.
\end{eqnarray*}
For  small $\gamma$,  by (\ref{s-con}) and (\ref{s-con2})  one has also \begin{eqnarray*}
 \frac{1}{2\|Df^{-1}\|}\le \min_{\|v\|_x''=1}\|D_yf(v)\|_{f(x)}''\le \max_{\|v\|_x''=1}\|D_yf(v)\|_{f(x)}''
\le 2\|Df\|.
\end{eqnarray*}
It follows that for   $\|v\|_x''=1$,   the angle $\angle''( D_{y}f(v),D_{x}f(v))$ with respect to $\|\cdot\|_{f(x)}''$ is less than   $\arctan(\xi)-\arctan(\varsigma\xi)$ for $a_{\xi}$ small enough. We conclude that $D_yf\left(Q_{\|\|_x''}(E^u_x), \xi\right)\subset Q_{\|\|_{f(x)}''}\left(E^u_{f(x)},\xi\right)$ for any $y\in B(x,a_{\xi}\gamma_k)$. We prove  similarly  the  cone invariance property for the center stable direction. \\

\item[(iii)]To prove the last item observe first that using again the domination property at $x$  we get  $\left|\frac{\| \varcurlywedge^{l}D_xf(\iota(G))\|_{f(x)}''}{\|\varcurlywedge^{l}D_xf(\iota(E^u_x))\|_{f(x)}''}-1\right|\leq 1-e^{-\gamma/2}$ for all $l$-planes $G\subset Q_{\|\|_x''}(E^u_x,\xi)$ for $\xi>0$ small enough. As $f$ is $e^{\lambda''_u}$-expanding in the unstable direction with respect to $\|\cdot \|''$ we have  $$
\|\varcurlywedge^{l}D_xf(\iota(E^u_x))\|_{f(x)}''\geq e^{l\lambda_u''}.$$
Then arguing as above for $y\in B(x,a_\xi\gamma_{\kappa(x)}^{l})$,   we have by Lemma \ref{app} in the  Appendix and the Inequalities (\ref{mettt}) :
\begin{align*}\|\varcurlywedge^lD_xf-\varcurlywedge^lD_yf\|_{f(x)}''&\leq (4a_1e^{\kappa\left(f(x)\right)\gamma})^l\|\varcurlywedge^lD_xf-\varcurlywedge^lD_yf\|\leq a_\xi^{\alpha}.\end{align*}
Therefore we get for $a_{\xi}$ small enough :

 \begin{align*}\left| \frac{\|\varcurlywedge^{l}D_yf(\iota(G))\|_{f(x)}''}{\|\varcurlywedge^{l}D_xf(\iota(E^u_x))\|_{f(x)}''}-1\right|&\leq \frac{\left|\|\varcurlywedge^{l}D_yf(\iota(G))\|-\|\varcurlywedge^{l}D_xf(\iota(G))\|_{f(x)}''\right|}{\|\varcurlywedge^{l}D_xf(\iota(E^u_x))\|_{f(x)}''}\\[2mm]
 &+\left|\frac{\| \varcurlywedge^{l}D_xf(\iota(G))\|_{f(x)}''}{\|\varcurlywedge^{l}D_xf(\iota(E^u_x))\|_{f(x)}''}-1\right|\\
&\leq \frac{a_{\xi}^\alpha}{e^{l\lambda''_u}}+1-e^{-\gamma/2}\\[2mm]
&\leq 1-e^{-\gamma}.
\end{align*}
\end{itemize}
\end{proof}
From the domination structure $E^{cs}\oplus E^{u}$ in the norm $\|\cdot\|''_x$,    one  may build a family of   \textit{fake}                 center-stable manifolds as follows.

\begin{prop}\label{local manifolds}   With the notations of Lemma \ref{neighborhood}, for any $\xi>0$ small enough, there exist $b_{\xi}\in (0,a_\xi)$    and families  $\{\mathcal  W^{cs}_x:\,x\in \Lambda^*\}$ of $C^1$ manifolds satisfying
	\begin{itemize}
	\item[(i)] uniform size: for $x\in \Lambda_k$, $k\in \mathbb{N}$, 
there is a $C^1$ map $\phi_x:E^{cs}_x\rightarrow E^u_x$ such that $\mathcal  W^{cs}_x$ is locally given by the graph $\Gamma\phi_x:=\left\{(z, \phi_x(z)), \ z\in E^{cs}_x\right\}$ of $\phi_x$, i.e. $$\mathcal  W^{cs}_x=\exp_x(\Gamma \phi_x)\cap  B(x,    a_\xi\gamma_k); $$

	\item[(ii)] almost tangency: $T_y \mathcal W^{cs}_x$ lies in a cone of width $\xi$ of $E^{cs}_x$ in $\|\cdot\|_x''$ for any $y\in \mathcal W^{cs}_x$;
		\\

		\item[(iii)] local invariance: $f^{\pm}\mathcal W^{cs}_x(b_\xi\gamma_{\kappa(x)}) \subset \mathcal W^{cs}_{f^{\pm}(x)}$ with $\mathcal W^{cs}_x(\zeta)$ being the ball of radius $\zeta$ centered at $x$ inside $\mathcal W^{cs}_x$ with respect to the distance induced by $\|\cdot\|''_x$ on $\mathcal W^{cs}_x$.
	\end{itemize}
	
\end{prop}



\begin{proof} By taking  the exponential map at $x$ we can assume without loss
of generality that we are working in   $\mathbb{R}^{\mathsf{d}}$. Let $\xi$ and $a_\xi$ be as in Lemma \ref{neighborhood}. For any $x\in \Lambda^*$, we can extend $f\mid_{B(x,a_{\xi}\gamma_{\kappa(x)})}$ to a diffeomorphism $\tilde{f}_x: \mathbb{R}^{\mathsf d} \to \mathbb{R}^{\mathsf d} $ such that
	\begin{itemize}
		\item $\tilde{f}_x(y)=f(y)$\quad for $y\in B(x,a_{\xi}\gamma_{\kappa(x)})$;\\
		\item  $\|D_y\tilde{f}_x-D_xf\|''_x\leq 2 a_{\xi}^{\alpha}$\quad for $y\in \mathbb{R}^{\mathsf d}$.
	\end{itemize}
By taking $a_{\xi}$ smaller,  the properties  of  Lemma \ref{neighborhood}   hold with respect to   $\tilde{f}_x$ for all $y\in \mathbb{R}^{\mathsf d}$. 	Let $\Xi$ be the disjoint union given by $\Xi=\coprod_{x\in \Lambda^*} \{x\}\times \mathbb{R}^{\mathsf d}$ where $\Lambda^*$ is endowed with the discret topology.  Then $\tilde f=(\tilde{f}_x)_{x\in \Lambda^*}$ can be viewed as a map from $\Xi$ to itself by letting $\tilde{f}(x,v)= \left(f(x),\tilde f_x(v)\right)$. Note that the global splitting $\coprod_{x\in \Lambda^*}\{x\}\times \mathbb{R}^{\mathsf d}=\coprod_{x\in \Lambda^*} \{x\}\times (E^{cs}_x\oplus E^u_x)$ is dominated with respect to $\tilde{f}$.  By \cite{HPS} $\S5$,  we can obtain
	 a family $ \{ \mathcal{Y}^{cs}_x:\,x\in \Lambda^*\}$ of global $C^1$ submanifolds in $\mathbb{R}^{\mathsf d}$ which are $C^1$ graphs defined on $E^{cs}_x$ such that we have for all $x\in \Lambda^* $ :
	 \begin{eqnarray*} &&\{x\}\times \mathcal  Y^{cs}_x=\bigcap_{n=0}^{+\infty}\tilde{f}^{-n}\left(\{f^n(x)\}\times Q_{\|\|_x''}(E^{cs}_{f^n(x)}, \xi)\right),\\[2mm]
	 && \forall y\in \mathbb{R}^{\mathsf d},\ T_y\mathcal Y^{cs}_x \subset Q_{\|\|_x''}(E^{cs}_{x}, \xi).\end{eqnarray*}
	 In particular we get  $\tilde{f}^{\pm}(\{x\}\times \mathcal Y^{cs}_x) \subset\{f^{\pm}(x)\}\times  \mathcal Y^{cs}_{f^{\pm}(x)}$. Since we have $\tilde{f}\mid_{\{x\}\times B(x,a_{\xi}\gamma_{\kappa(x)})}=f\mid_{B(x,b_{\xi}\gamma_{\kappa(x)})}$, one concludes the proof by considering   $\mathcal W^{cs}_x=\mathcal Y_x\cap B(x,    a_{\xi}\gamma_{\kappa(x)})$ and taking much smaller $b_\xi$ than $a_{\xi}$. 

\end{proof}



\section{Tail entropy and local volume growth}\label{volume} Let $f:M\rightarrow M$ be a $\mathcal{C}^r$ diffeomorphism  with $r>1$ on a compact Riemannian manifold $(M,\|\|)$.
In this section,  we relate the Newhouse local entropy of an ergodic measure with the local volume growth of smooth \textit{unstable} discs. We begin with some definitions.
A $\mathcal{C}^{r}$ map $\sigma$,  from the unit square  $[0,1]^k$ of $\R^k$ to $M$, which is a diffeomorphism onto its image, is called a $\mathcal{C}^r$ $k$-disc.  The $\mathcal{C}^r$ size of $\sigma$ is defined as
$$\|\sigma\|_{r}=\sup\{\|D^q\sigma\|:\,\,q\le r,\,\, q\in \mathbb{R}^+\},$$
where $\|D^q\sigma\|$ denotes the $(q-[q])$-H\"older norm of $D^{[q]}\sigma$ for $q\notin \mathbb{N}$ and the usual supremum norm of the derivative $D^q\sigma$ of order $q$ for $q\in \mathbb{N}$.

For any $\mathcal{C}^1$ smooth $k$-disc $\sigma$ and for any $\chi>0$, $1\gg \gamma>0$, $C>1$ and $n\in \mathbb{N}$, we consider the set $\mathcal{H}^n_f(\sigma,\chi,\gamma,C)$ of points of $[0,1]^k$ whose exponential growth of the induced map  on the $k$-exterior tangent bundle  is almost equal to $\chi$:

$$\mathcal{H}^n_f(\sigma,\chi,\gamma,C):=\left\{t\in [0,1]^k \ :  \ \forall 1\leq j\leq n-1, \ C^{-1}e^{(\chi-\gamma)j}\leq \| \varcurlywedge^kD_t\left(f^j\circ\sigma\right)\|\leq Ce^{(\chi+\gamma)j}\right\}.$$

 For $\Gamma\subset [0,1]^k$,  we  also denote by $|\sigma_{|\Gamma}|$ the $k$-volume of $\sigma$ on $\Gamma$, i.e. $|\sigma_{|\Gamma}|=\int_{\Gamma}\|\varcurlywedge^kD_t\sigma\|\, d\lambda(t)$,  where $d\lambda$ is the Lebesgue measure on $[0,1]^k$. Then given  $\chi>0$, $1\gg\gamma>0$, $C>1$, $x\in M$, $n\in \N$ and $\vep >0$,   we define the local  volume growth of $\sigma$ at $x$ with respect to these parameters as follows :
$$ V_x^{n,\vep}\left(\sigma \middle| \chi,\gamma,C\right):= \left| (f^{n-1}\circ \sigma)_{|\Delta_n}\right|$$
$$\text{ with } \Delta_n:=\mathcal{H}^n_f(\sigma,\chi,\gamma,C) \cap \sigma^{-1}B_n(x,\vep,f).$$

\begin{prop}\label{ens}
Let $\nu\in \mathcal{M}_{erg}(f)$ with $l=l_u(f, \nu)\geq 1$. Then for any $\varepsilon>0$, $1>\eta>0$ and  $\gamma>0$,  there exist
a Borel set $F_\eta$ with $\nu(F_\eta)>\eta$ and a constant $C>1$, such that for all $\delta>0$,   all $n$ large enough and  all  $x\in F_{\eta}$ :
\begin{eqnarray*}\label{vvolume}s(n,\delta, B_n(x,\vep,f)\cap F_{\eta})\leq e^{\gamma n}\sup_{\stackrel{\sigma\, l\text{-disk}}{\text{with}\,\|\sigma\|_{r}\le 1}}V_x^{n,2\vep}\left(\sigma \middle| \sum_i\lambda_i^+(\nu),\gamma,C\right).
\end{eqnarray*}  
\end{prop}

  In fact   the $l_{u}$-discs can be chosen to be affine through the exponential map (see  the proof of Proposotion \ref{ens} below).  Let $v_k^*(f, \vep)$ denote the local volume growth of $k$-disks :
  $$v^*_k(f,\vep)=\limsup_n\frac{1}{n}\sup_{x\in M}\sup_{\stackrel{\sigma\, k\text{-disk}}{\text{with}\,\|\sigma\|_{r}\le 1}}\log \left| (f^{n-1}\circ \sigma)_{|\sigma^{-1}B_n(x,\vep,f)}\right|.$$

  S. Newhouse \cite{New1, New} proved that the Newhouse local entropy $h^*(f,\nu,\vep)$ of an ergodic measure is less than or equal to the local volume growth of center-unstable dimension. As a direct consequence of Proposition \ref{ens}, we improve this estimate by  considering the local volume growth of unstable dimension.

  \begin{cor}\label{abo}
With the above notations,
$$\forall \vep>0 \ \forall \nu \in \mathcal{M}_{erg}(f), \ h^*(f,\nu,\vep)\leq v^*_{l_u(f,\nu)}(f,2\vep).$$
  \end{cor}

Such an inequality was established by K. Cogswell in \cite{cog} between the Kolmogorov-Sinai entropy and the global volume growth of unstable discs (in particular Cogswell's main result implies Corollary \ref{abo}  for $\varepsilon$ larger than the diameter of $M$).

  \begin{remark}For any $\nu \in \mathcal{M}_{erg}(f)$,   let us denote by $l_{cu}(f,\nu)$ the number of nonnegative Lyapunov exponents of $\nu$. The following estimate  is shown in \cite{New} :
\begin{eqnarray*}\hspace{0,8cm} \forall \vep>0  \ \forall\nu \in \mathcal{M}_{erg}(f), \ h^*(f,\nu,\vep)\leq\sup_{\stackrel{\sigma \, l_{cu}(f,\nu)\text{-disk}}{ with \,\|\sigma\|_{r}\le 1}}\limsup_n\frac{1}{n}\sup_{x\in M}\log \left| (f^{n-1}\circ \sigma)_{|\sigma^{-1}B_n(x,2\vep,f)}\right|.\end{eqnarray*}Observe the right-hand
 side term differs from the local volume growth $v^*_{l_{cu}(f,\nu)}(f,2\vep)$ as we invert the supremum in $\sigma$ with the limsup in $n$.  We do not know if the above inequality still holds true for $l_u$ in place of $l_{cu}$.
  \end{remark}

  We prove now Proposition \ref{ens} which is the key new tool to prove the existence of symbolic extensions in dimension 3  combining with the approach developed in \cite{Burguet12}.

 \begin{proof}[Proof of Proposition \ref{ens}]
 Consider  $\nu\in \mathcal{M}_{erg}(f)$ with $l=l_u(f,\nu)\geq 1$.
Let $ 0<\gamma\ll \lambda_{u}:=\lambda_{l}(f,\nu)$ in the nonuniformly hyperbolic estimates of Section 2.
Fix $\eta\in (0,1)$ and  $k\in \mathbb{N}$ with  $\nu(\Lambda_k(\lambda_u,\gamma))>\eta$.   There is  a subset $F_\eta$ of $\Lambda_k=\Lambda_k(\lambda_u,\gamma)$ with $\nu(F_\eta)>\eta$ such that $\frac{1}{n}\|\varcurlywedge^{l}D_y(f^n|_{E^u_y})\|$ is converging uniformly in $y\in F_\eta$ to $\sum_i\lambda_i^+(f,y)=\sum_i\lambda_i^+(f,\nu)$ when $n$ goes to $+\infty$.
Let $\vep\in (0,1)$ and $\vep_k<\vep$ to be precised.
For any given $\hat x\in F_\eta$,  $0<\delta<\vep$, let $E_n$ be a maximal $(n,\delta)$-separated set in $d$ for $f$ in $B_n(\hat x,\vep, f)\cap F_\eta$. There exists $x\in E_n$ such that $E'_n=E_n\cap B(x,\vep_k)$ satisfies $\sharp E'_n \geq A_1\left(\frac{\varepsilon_k}{\varepsilon}\right)^{\mathsf d}\sharp E_n$ for some universal  constant $A_1$. 
Since we only deal with the local dynamics around the orbit of $x$, we can assume without loss of generality that we are working in $\mathbb{R}^\mathsf{d}$ by taking the exponential map at $x$. 
 Take $0<\vep_k<(a_1e^{k\gamma})^{-1}$ so small that $B(x, \vep_k)\subset B_x''(x, 2a_1e^{k\gamma}\vep_k) \subset B(x,\vep)$ and consider $$\hat{\mathcal W}_x^{cs}=(x+E^{cs}_x)\cap B''_x(x, 2a_1e^{k\gamma}\vep_k) .$$
For $\theta_{n}=\beta_k e^{-n(4\gamma+l\gamma')}$  with $\beta_k=\beta_k(\delta)$ to be precised we let $\mathcal{A}^{cs}$ be  a $\theta_{n}$-net   of $\hat{\mathcal W}_x^{cs}$ for $d''_x$ satisfying  $\sharp\mathcal{A}^{cs}\le A_2\theta_n^{-\dim E^{cs}}= A_2\theta_n^{-(\mathsf d-l)}$ for some  universal constant $A_2$. This means that any point of $\hat{\mathcal W}_x^{cs}$ is within a distance $\theta_n$ of $\mathcal{A}^{cs}$ for $d''_x$. For any $z\in \mathcal{A}^{cs}$,  denote $$I_z= \{z+v: \|v\|_x''\le 4a_1e^{k\gamma}\vep_k,\, v\in E_x^u\}.$$ 

For $y\in B''_x(x,2 a_1e^{k\gamma}\vep_k) $ we let   $y=y_{cs}+y_u$ with  $y_{cs}\in x+E^{cs}_x$ and  $y_u\in E^u_x$.  Observe that $E^{cs}_x$ and $E^u_x$ are orthogonal in $<,>_x''$,  thus $y_{cs}$ lies in $ \hat{\mathcal W}_x^{cs}$ and there exists $z_y\in  \mathcal{A}^{cs}$ with $ \|y_{cs}-z_y\|''_x<\theta_n$. Therefore,   when $y$ also lies in $\Lambda_k$ we get  :
\begin{align*}
\|y_{cs}-z_y\|''_y&\leq 2a_1e^{k\gamma}\|y_{cs}-z_y\|\\[2mm]
&\leq 4a_1e^{k\gamma}\|y_{cs}-z_y\|''_x\\[2mm] 
&\leq 4a_1e^{k\gamma}\theta_n.
\end{align*}

   For small $\xi\in(0, \frac{1}{4})$, let  $b_{\xi}>0$  be as in Lemma \ref{neighborhood} and Proposition \ref{local manifolds}. Since the distributions $E^{cs}$  and $E^u$ are continuous on $\Lambda_k$, we may choose $\varepsilon_k$ and $\beta_k$ so small that for any $y\in E'_n$:
   \begin{itemize}
   \item  the set $\left([y_{cs},z_y]+E^u_x\right)\cap \mathcal W^{cs}_y$ defines  a  
graph $\Gamma_{\phi_y}$ of a $C^1$ function $\phi_y: [y_{cs},z_y]\subset E^{cs}_x \rightarrow E^u_x$,
   \item $E^{cs/u}_x\subset Q_{\|\|}\left(E^{cs/u}_y,\frac{\xi}{4a_1e^{k\gamma}}\right)\subset Q_{\|\|_y''}\left(E^{cs/u}_y,\xi\right)$, these cones being defined with respect to the splitting $E^{cs}_y\oplus E^u_y$.
   \end{itemize}

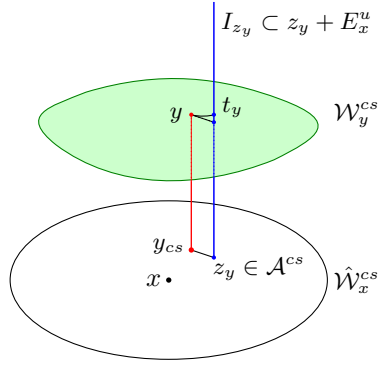
\begin{figure}[!h]
	
\begin{tikzpicture}

\draw (0,-1) ellipse (60pt and 30pt);

\filldraw[fill=green!20!white, draw=green!50!black] (-1.5,0.8) .. controls (-2,1.2) and (-1.6,1.3) .. (-1.1,1.5) .. controls (-0.1,1.8) and (0.7,1.7) .. (1.5,1.4)
.. controls(2.3,1.1) and (2,0.9)..(1.4, 0.6)..controls (0.4,0.2) and (-0.7,0.2)..(-1.5,  0.8);

\filldraw [color=blue!100, densely dotted](0.6,0.36) -- (0.6,1.1);

\filldraw [color=blue!100](0.6,-0.7) -- (0.6,0.36);

\filldraw [color=blue!100](0.6,1.1) -- (0.6,2.7);

\filldraw [color=red!100, densely dotted](0.3, 0.31) -- (0.3,1.2);

\filldraw [color=red!100](0.3,-0.6) -- (0.3,0.31);

\draw(0.3,-0.6) -- (0.6,-0.7);

\filldraw [densely dotted](0.3,1.2) -- (0.6,1.1);

\draw  (0.3,1.2) ..controls(0.4, 1.17)  and (0.55, 1.18).. (0.6,1.2);

\fill (0,-1) circle (1pt);
\node at (-0.2, -1){$x$};

\fill[color=blue] (0.6,-0.7) circle (0.8pt);
\node at (1.2, -0.85){{\small $z_y\in \mathcal{A}^{cs}$}};

\fill [color=red](0.3,-0.6) circle (1pt);

\node at (0, -0.55){{\small $y_{cs}$}};

\fill[color=red] (0.3,1.2) circle (0.8pt);
\node at (0.1, 1.2){{\small $y$}};

\fill[color=blue] (0.6,1.2) circle (0.8pt);
\node at (0.85, 1.3){{\small $t_y$}};

\fill[color=blue] (0.6,1.1) circle (0.8pt);

\node at (2.5, 1.2){{\small $\mathcal W_y^{cs}$}};

\node at (1.7, 2.4){{\small $ I_{z_y}\subset z_y+E^u_x$}};

\node at (2.5, -1){{\small $\hat{\mathcal W}_x^{cs}$}};

\end{tikzpicture}

	\caption{The transverse intersection at $t_y
$ of $I_{z_y}$  and $\mathcal W_y^{cs}$ for $y\in E'_n$.}
\end{figure}

   Let $\theta_y:[0,1]\rightarrow  E^{cs}_x+E^u_x$ be the reparametrization of the graph of $\phi_y$ given by 
   $$\forall t\in [0,1], \ \theta_y(t)= y_{cs}+t(z_y-y_{cs})+\phi(y_{cs}+t(z_y-y_{cs}).$$
 Note that $\theta_y(0)=y$ and $\theta_y(1)$ is the intersection point of $I_{z_y}$ and $\mathcal W_y^{cs}$. To simplify the notations we let $t_y:=  \theta_y(1)$. It follows from  the almost tangency property of center-stable fake manifolds stated in Proposition \ref{local manifolds} (ii) that 
    \begin{equation}\label{ein}\theta'(t)\in Q_{\|\|''_y}(E^{cs}_y,\xi).\end{equation}
Moreover   we have  \begin{eqnarray}
\label{zzwei}z_y-y_{cs}&\in E^{cs}_x \subset Q_{\|\|''_y}(E^{cs}_y,\xi),\\[2mm]
\label{ddrei}D_{y_{cs}+t(z_y-y_{cs})}\phi(z_y-y_{cs})&\in E^u_x\subset Q_{\|\|''_y}(E^{u}_y,\xi).
\end{eqnarray}

From the above properties (\ref{ein}), (\ref{zzwei}), (\ref{ddrei}) and $\xi<\frac{1}{4}$, one deduces after an easy computation that $\|\theta'(t)\|''_y\leq 3 \|z_y-y_{cs}\|''_y$ for all $t\in [0,1]$. For $w\in \Lambda^*$ let $d''_{\mathcal W^{cs}_w}$ be the distance induced respectively   by $\|\|''_w$  on $\mathcal W^{cs}_w$. We have 
\begin{align*}
d''_{\mathcal W^{cs}_y}(y,t_y)&\leq \int_{[0,1]}\|\theta'(t)\|''_y\, dt,\\[2mm]
&\leq 3\|y_{cs}-z_y\|''_y,\\[2mm]
&\leq 12a_1e^{k\gamma}\theta_n.
\end{align*}

Consequently by the local invariance of  center-stable manifolds stated in  Proposition \ref{local manifolds} (iii) we get for all   $0< j\le n$  :
\begin{eqnarray*}d''_{\mathcal W^{cs}_{f^j(y)}}\left(f^j(y), f^j(t_y)\right)&\le &
\|Df\mid_{T\mathcal W^{cs}_{f^{j-1}(y)}}\|''_{f^{j-1}(y)}
d''_{\mathcal W^{cs}_{f^{j-1}(y)}}\left(f^{j-1}(y),f^{j-1}(t_y)\right),
\end{eqnarray*}
and then by Lemma \ref{neighborhood} (i),
\begin{eqnarray*}
d''_{\mathcal W^{cs}_{f^j(y)}}\left(f^j(y), f^j(t_y)\right)&\le &
e^{4\gamma}
d''_{\mathcal W^{cs}_{f^{j-1}(y)}}\left(f^{j-1}(y),f^{j-1}(t_y)\right). 
\end{eqnarray*}
After an immediate induction we obtain for all   $0\le j\le n$ :
\begin{eqnarray*}
d''_{\mathcal W^{cs}_{f^j(y)}}\left(f^j(y), f^j(t_y)\right)&\le & e^{4j\gamma}d''_{\mathcal W^{cs}_y}(y,t_y),\end{eqnarray*}
 and therefore 
 \begin{eqnarray*}
d''_{\mathcal W^{cs}_{f^j(y)}}\left(f^j(y), f^j(t_y)\right)
& \le & 12a_1e^{k\gamma} e^{4n\gamma}\theta_n\\[2mm]
&\le & 12a_1e^{k\gamma}\beta_ke^{-nl\gamma'},\\[2mm]
d''_{\mathcal W^{cs}_{f^j(y)}}\left(f^j(y), f^j(t_y)\right)&\le & 12a_1e^{k\gamma}\frac{\beta_k}{\gamma_{\kappa(y)}}\gamma_{\kappa(f^j(y))}.
\end{eqnarray*}
 As $y$ belongs to $E'_n\subset \Lambda_k$ we have $\kappa(y)\leq k$. Therefore we get  for $\beta_k\leq \frac{b_\xi\gamma_k}{48a_1e^{k\gamma}}$ :
\begin{eqnarray*}
\forall\,0\le j\le n, \ d(f^j(t_y),f^j(y))&\leq& 2 d''_{\mathcal W^{cs}_{f^j(y)}}\left(f^j(y), f^j(t_y)\right)\\[2mm]
&\leq& \frac{b_\xi}{2}\gamma_{\kappa(f^j(y))}, 
\end{eqnarray*}
and we have similarly for $\beta_k<\frac{\delta}{48a_1e^{k\gamma}} $ :

 \begin{align*}
\forall\,0\le j\le n, \ d(f^j(t_y),f^j(y))&\leq \delta/4,\\[2mm]
\text{i.e.} \ t_y \in B_n(y,\delta/4,f)&.
\end{align*}

  For $y\in E'_n$ we let now
\begin{eqnarray*} W_n(t_y)&:=&
	\bigcap_{j=0}^{n-1}f^{-j}\left(B_{f^j(y)}''(f^{j}(t_y), \frac{\delta}{8} e^{-lj\gamma'})\right)\bigcap I_z,\\[2mm]
&\subset &B_n(t_y,\delta/4,f),\\[2mm]
&\subset & B_n(y,\delta/2,f).
\end{eqnarray*}
  As  $E'_n$ is  $(n,\delta)$-separated, the sets $\left(W_n(t_y)\right)_{y\in E'_n}$ are pairwise disjoint.


 For $\delta$ small enough (depending only on $k$),  for any $j=0,\cdots, n-1$, the ball $ B_{f^j(y)}''\left(f^{j}(t_y),    \frac{\delta}{8} e^{-lj\gamma'}\right)$ is contained in $B\left(f^j(y), b_\xi\gamma_{\kappa(f^j(y))}^{l}\right)$, since  $d(f^j(t_y), f^j(y))\le \frac{b_\xi}{2}\gamma_{\kappa(f^j(y))}^{l}$. 
 Let  $(e_x^i)$ be an  orthonormal basis of $E^u_x$ with respect to $\|\cdot\|''_x$. We consider the affine reparametrization of $I_z$,  $z\in \mathcal{A}^{cs}$, given by   $\sigma_z:[0,1]^{l_\nu}\rightarrow M$, $(t_i)_i\mapsto z+\sum_i(t_i-1/2)4a_1e^{k\gamma}\vep_ke_x^i$. 
 Noting  that   $E_x^u\in Q_{\|\|_y''}(E^{u}_y, \xi) $,  by Lemma \ref{uniform size} (ii),   for any $\tau\in \sigma_z^{-1}W_n(t_y)$ and for any $0\leq j\leq n$,
 the vector space $D_{\sigma_z(\tau)}f^j(E_x^u)$ lies in $Q_{\|\|''_{f^j(y)}}\left(E^u_{f^j(y)}, \xi\right)$. Then by Lemma \ref{uniform size} (iii)
 we get 
 \begin{eqnarray*}
\limsup_n\frac{1}{n} \log\|\varcurlywedge^{l}D_{\sigma_z(\tau)}f^n|_{E^u_x}\|''_y&=& \limsup_n\frac{1}{n}\sum_{j=0}^{n-1} \log\|\varcurlywedge^{l}D_{f^j\circ\sigma_z(\tau)}f|_{D_{\sigma_z(\tau)}f^j(E_x^u)}\|''_{f^j(y)},\\
&\leq & \limsup_n\frac{1}{n}\sum_{j=0}^{n-1} \log\|\varcurlywedge^{l}D_{f^j(y)}f|_{E^u_{f^j(y)}}\|''_{f^j(y)}+\gamma,\\
&= &\limsup_n\frac{1}{n} \log\|\varcurlywedge^{l}D_{y}f^n|_{E^u_y}\|''_y +\gamma.
\end{eqnarray*}
 
 Noting that $f^n(y)\in \Lambda_{k+n}$, we have by  the Inequalities (\ref{mettt})  
 $$\forall v\in T_{f^n(\sigma_z(\tau))}M, \ \frac{\|v\|}{2}\leq \|v\|''_{f^n(y)}\leq 2a_1e^{(k+n)\gamma}\|v\|.$$
 
 Then it follows from Lemma \ref{app} in the  Appendix  that 
 \begin{align*}
\limsup_n\frac{1}{n} \log\|\varcurlywedge^{l}D_{\tau}(f^n\circ \sigma_z)\|&=
\limsup_n\frac{1}{n} \log\|\varcurlywedge^{l}D_{\sigma_z(\tau)}f^n|_{E^u_x}\|,\\[2mm]
&\leq \limsup_n\frac{1}{n} \log\|\varcurlywedge^{l}D_{\sigma_z(\tau)}f^n|_{E^u_x}\|''_y+l\gamma,  \\[2mm]
&\leq \limsup_n\frac{1}{n} \log\|\varcurlywedge^{l}(D_{y}f^n|_{E^u_y})\|''_y + (l+1)\gamma,\\[2mm]
&\leq \limsup_n\frac{1}{n}\log \|\varcurlywedge^{l}(D_{y}f^n|_{E^u_y})\| +(2l+1)\gamma,\\[2mm]
&= \sum_i\lambda_i^+(f,\nu)+(2l+1)\gamma.
\end{align*}

 Similarly we also get :
  $$\liminf_n\frac{1}{n} \log\|\varcurlywedge^{l}D_{\tau}(f^n\circ \sigma_z)\|\geq \sum_i\lambda_i^+(f,\nu)-(2l+1)\gamma.$$

Moreover the above limsup and liminf are uniform in $y\in E'_n$  and $\tau \in \sigma_z^{-1}W_n(t_y)$.  Therefore for some $C>1$ we have  for $n$ large enough, $$\sigma_z^{-1}W_n(t_y)\subset  \mathcal{H}_f^{ n} \left(\sigma_z, \sum_i\lambda_i^+(f,\nu), (2l+2)\gamma, C\right).$$

By using Lemma \ref{uniform size} and classical arguments of graph transform, the set $f^j(W_n(t_y))$ for $0\le j\le n-1 $ defines a graph of a function from $B_{f^j(y)}''\left(f^{j}(t_y),\frac{\delta}{8} e^{-lj\gamma'})\right)\cap E^u_{f^j(y)}$ to   $E^{cs}_{f^j(y)}$.  Therefore the $l$-volume of $f^{n-1}(W_n(t_y))$ with respect to $\|\cdot \|''_{f^{n-1}(y)}$ satisfies  \begin{eqnarray}\label{eins}\left |f^{n-1}(W_n(t_y))\right|''_{f^{n-1}(y)}&\geq & c_{l}\delta^l e^{-{l^2}(n-1)\gamma'},\end{eqnarray}
for some   universal constant $c_{l}$.   By applying again Lemma \ref{app} in the Appendix we obtain :
\begin{eqnarray}\label{zwei}\left |f^{n-1}(W_n(t_y))\right| &\geq &(4a_1e^{(k+n-1)\gamma})^{-l}\left |f^{n-1}(W_n(t_y))\right|''_{f^{n-1}(y)}
 ,\end{eqnarray}
where $\left |f^{n-1}(W_n(t_y))\right|$ denotes the $l$-volume of 
$f^{n-1}(W_n(t_y))$ with respect to the Riemannian norm $\|\cdot \|$ on $M$. For  $z\in \mathcal A^{cs}$ we   let $$\Gamma_z:=\{y\in E'_n, \ z_y=z\}$$ and $$\Delta_n^z:= \mathcal{H}_{f}^{n} \left(\sigma_z, \sum_i\lambda_i^+(f,\nu),        (2l+2)\gamma, C \right)\cap\sigma_z^{-1} B_n(x, 2\vep, f).$$
As the sets $W_n(t_y)$, $y\in \Gamma_z$, are pairwise disjoint subsets of $\sigma_z(\Delta_n^z)$ we have :
\begin{eqnarray}\label{drei}
 \left|(f^{n-1}\circ\sigma_z)_{|\Delta_n^z}\right| &\geq &  \sum_{y\in \Gamma_z}\left|f^{n-1}(W_n(t_y))\right|.
	\end{eqnarray}
By combining the inequalities (\ref{eins}), (\ref{zwei}), (\ref{drei}) we obtain 

\begin{eqnarray*}\left|(f^{n-1}\circ\sigma_z)_{|\Delta_n^z}\right| &\geq & (4a_1e^{(k+n-1)\gamma})^{-l}\sum_{y\in \Gamma_z}\left|f^{n-1}_{|W_n(t_y)}\right|_{f^{n-1}(y)}'',\\
&\geq &   c_{l}\delta^l e^{-{l^2}(n-1)\gamma'} (4a_1e^{(k+n-1)\gamma})^{-l}\cdot\sharp\Gamma_z.
\end{eqnarray*} With the notations introduced at the beginning of Section 3, we have therefore for some constant $D$ independent of $n$ and $\hat x\in F_\eta$,
$$\#\Gamma_z   \le  De^{n(l\gamma+l^2\gamma')} V_x^{n,2\vep}\left(\sigma_z \middle| \sum_i\lambda_i^+(f,\nu), (2l+2)\gamma,C\right). $$

 By letting $\mathcal{F}_{n,\delta}:=\{\sigma_z,\ z\in \mathcal{A}^{cs}\}$,  we get for all $\hat x\in \Lambda_k$ and some constants, all denoted by $D$ and  independent of $n$ and  $\hat x\in F_\eta$ :
 \begin{eqnarray*} s\left(n,\delta, B_n(\hat x,\vep, f)\right)&=&\sharp E_n,\\[2mm]
 &\leq & D\sharp E'_n,\\[2mm]
 &\leq & D \sum_{z\in  \mathcal{A}^{cs} }\sharp \Gamma_z,\\[2mm]
 &\leq  &De^{n(l\gamma+l^2\gamma')}\sharp\mathcal{A}^{cs}\sup_{\sigma\in \mathcal{F}_{n,\delta}}V_x^{n,2\vep}\left(\sigma_z \middle| \sum_i\lambda_i^+(f,\nu),(2l+2)\gamma,C\right), \\[2mm]
  &\leq & De^{n((\mathsf d-l)(4\gamma+l\gamma')+l\gamma+l^2\gamma')}\sup_{\sigma\in \mathcal{F}_{n,\delta}}V_x^{n,2\vep}\left(\sigma_z \middle| \sum_i\lambda_i^+(f,\nu),(2l+2)\gamma,C\right).
\end{eqnarray*}
This concludes the proof of Proposition \ref{ens} as $\gamma$ and thus  $\gamma'=\alpha^{-1}\gamma$ may be chosen arbitrarily small.
 \end{proof}

\section{Proof of Main Theorem \ref{maintheorem}}

By Proposition \ref{ens} Newhouse local entropy of an ergodic measure with one positive Lyapunov exponent is bounded from above by the local volume growth of curves. This volume growth may be controled by using the Reparametrization Lemma of \cite{Burguet12}. Following straightforwardly the proof of the Main Proposition in \cite{Burguet12} we get :

\begin{prop}\label{res}Let $f$ be a $\mathcal{C}^r$ diffeomorphism  with $r>1$ on a Riemannian manifold  $M$ and $\mu\in \mathcal{M}_{inv}(f)$. For all $\gamma>0$, there exist  $m_\mu, k_{\mu}\in \mathbb{N}^*$  such that for $\nu\in \mathcal{M}_{erg}(f)$ close enough to $\mu$ with $l_u(f,\nu)=1$,  we have
$$h_{m_\mu, k_\mu}^{New}(f, \nu)\leq \frac{\lambda^+_1(f,\mu)-\lambda_1^+(f,\nu)}{r-1}+\gamma.$$
\end{prop}
From  the criterion in Theorem \ref{SEX},  for proving  the  Main Theorem,   we need  consider all ergodic measures with any possible $l_{u}$.
Actually,  the Main Theorem  is  obtained  from the  following Proposition  by  applying Theorem   \ref{SEX} with  the   upper semicontinuous affine function $E:=\frac{1}{r-1}\sum_{i=1,2}\lambda_i^+(f,\cdot)$.

\begin{prop}\label{res2}Let $f$ be a $\mathcal{C}^r$ diffeomorphism with $r>1$ on a 3-dimensional Riemannian manifold  $M$ and $\mu\in \mathcal{M}_{inv}(f)$. For all $\gamma>0$, there exist an entropy structure $(h_k)_k$ and  $k_\mu\in \mathbb{N}$ such that for $\nu\in \mathcal{M}_{erg}(f)$ close enough to $\mu$,  we have
	\begin{equation*}\label{12}h_{k_\mu}(f,\nu)\leq \frac{\sum_{i=1,2}\lambda^+_i(f,\mu)-\sum_{i=1,2}\lambda_i^+(f,\nu)}{r-1}+\gamma.\end{equation*}
\end{prop}

In other terms,  $E:=\frac{1}{r-1}\sum_{i=1,2}\lambda_i^+(f,\cdot)$ satisfies Inequaliy (\ref{fon}) for a 3-dimensional $\mathcal{C}^r$ diffeomorphism $f$ with $r>1$.

\begin{proof}[Proof of Proposition \ref{res2}]
 Fix $\mu\in \mathcal{M}_{inv}(f)$.  By the upper semicontinuity of $ \sum_{i=1,2}\lambda^+_i(f,\cdot)$,  lower semicontinuity of $\lambda_3(f, \cdot)$ and continuity of the integral of logarithm  for Jacobian,  when $\nu$ is close enough to $\mu$, one has
 \begin{eqnarray}\label{upper}  \sum_{i=1,2}\lambda^+_i(f,\mu)- \sum_{i=1,2}\lambda^+_i(f,\nu)&\ge& -\frac{(r-1)\gamma}{2},\\[2mm]
\label{lower} \lambda_3(f,\mu)-\lambda_3(f, \nu)&\le&  \frac{\gamma}{2},\\[2mm]
\label{con} \big{|}  \int \log \Jac(f)\, d\nu-  \int \log \Jac(f)\, d\mu \big{|}&\le& (r-1)\gamma.
\end{eqnarray}
 Hence,  if $h_{\nu}(f)\le \gamma/2$, from $h_{m, k}^{New}(f,\nu)\le h_{\nu}(f)$ for any $m, k$,  by (\ref{upper}) we get
 \begin{equation}\label{positive}h_{m, k}^{New}(f,\nu)\leq \frac{\sum_{i=1,2}\lambda_i^+(f,\mu)-\sum_{i=1,2}\lambda_i^+(f,\nu)}{r-1}+\gamma.\end{equation}
 Next we assume $h_{\nu}(f)>\gamma/2$.
 By Ruelle inequality \cite{Ruelle},  it holds that  $\min\left(l_u(f,\nu),l_u(f^{-1}, \nu)\right)=1$. Applying  Proposition \ref{res} to  $f^{\pm}$,  there exist  $m_\mu^{\pm}, k_{\mu}^{\pm}\in \mathbb{N}$  such that for any  $\nu\in \mathcal{M}_{erg}(f)$ close enough to $\mu$ with $l_u(f^{\pm},\nu)=1$,
 $$h_{m_\mu^{\pm}, k_{\mu}^{\pm}}^{New}(f^{\pm},\nu)\leq \frac{\lambda_1^+(f^{\pm},\mu)-\lambda_1^+(f^{\pm},\nu)}{r-1}+\gamma.$$
If  $l_u(f,\nu)=1$, then  $\sum_{i=1,2}\lambda_i^+(f,\nu)=\lambda_1^+(f,\nu)$, thus  by the above inequality,  (\ref{positive}) holds with respect to $m_\mu^+, k_{\mu}^+$.
  If
 $l_u(f^{-1},\nu)=1$,   then $\lambda_3(f, \nu)\le -h_{\nu}(f^{-1})=-h_{\nu}(f)<-\gamma/2$,  which implies  $\lambda_3(f,\mu)<0$ by (\ref{lower}).   Thus,
\begin{eqnarray*}
 \lambda_1^+(f^{-1},\mu)-\lambda_1^+(f^{-1},\nu)&=& \lambda^-_3(f,\nu)-\lambda^-_3(f,\mu)\\[2mm]
 &=&\lambda_3(f,\nu)-\lambda_3(f,\mu).
 \end{eqnarray*}
 Noting  that $ \int \log \Jac(f) \, d\tau =\sum_{i=1,2,3}\lambda_i(f,\tau)$ for any $\tau\in \mathcal{M}_{inv}(f)$,  by (\ref{con})
we finally get
 \begin{eqnarray*}
  \lambda_1^+(f^{-1},\mu)-\lambda_1^+(f^{-1},\nu)&= & \sum_{i=1,2}\lambda_i(f,\mu)-\sum_{i=1,2}\lambda_i(f,\nu)+ \int \log \Jac(f)\, d\nu-  \int \log \Jac(f)\, d\mu \\[2mm]
 &\le &  \sum_{i=1,2}\lambda^+_i(f,\mu)-\sum_{i=1,2}\lambda^+_i(f,\nu)+(r-1)\gamma
 \end{eqnarray*}
 and therefore
\begin{equation}\label{ddf}
h_{m_\mu^{-}, k_{\mu}^{-}}^{New}(f^{-1},\nu)\leq \frac{\sum_{i=1,2}\lambda^+_i(f,\mu)-\sum_{i=1,2}\lambda^+_i(f,\nu)}{r-1}+\gamma.
\end{equation}
By Lemma 2 in \cite{Burguet12},  the sequence $(\underline{h_k})_k:=(\min(h_{m_\mu^+, k}^{New}(f,\cdot), h_{m_\mu^-,k}^{New}(f^{-1}\cdot)))$ defines an entropy structure.  Combining  (\ref{positive})   for $l_{u}(f,\nu)= 1$ and (\ref{ddf})  for $l_{u}(f^{-1},\nu)= 1$, we conclude   the proof by considering the  entropy structure  $(\underline{h_k})_k$.
 \end{proof}

\begin{remark}
For a local diffeomorphism $f:M\rightarrow M$,  the following local Ruelle inequlity holds \cite{BF}\cite{CLY} :   there exists a scale $\varepsilon>0$ such that $h^*(f,\mu,\vep)\leq \min\left(\sum_j\lambda^+_j(f,\mu), - \sum_j\lambda^-_j(f,\mu)\right)$  for any $\mu\in \mathcal{M}_{inv}(f)$. In particular in dimension $3$, any invariant measure with positive Newhouse local entropy admits at least one positive and one negative Lyapunov exponent. As the proofs of Main Theorem  and Proposition \ref{ens} are just local they apply verbatim in the context of local $3$-dimensional diffeomorphism. 
\end{remark}
\appendix
\section{}
Let $E$ and $F$ be two finite dimensional vector spaces of dimension $k$. We endow 
$E$ (resp. $F$) with two Euclidean norms $\|\cdot\|_E$ and $\|\cdot\|'_E$ (resp. $\|\cdot\|_F$ and $\|\cdot\|'_F$). We consider the associated Euclidean structures on $\varcurlywedge^k E$ (resp. $\varcurlywedge^k F$). Let $A:E\rightarrow F$ be an invertible  linear map and $\varcurlywedge^kA$ the induced map on the $k$-exterior powers. We denote by $\|\varcurlywedge^kA\|$ and  $\|\varcurlywedge^kA\|'$ the associated subordinated norms.

\begin{lemma}\label{app}With the above notations. Assume that we have for some constants $C_E,C_F\geq 1$ and $D_E,D_F\leq 1$ :
\begin{eqnarray*}
\forall v\in E,  & D_E\|v\|_E \leq \|v\|_E' \leq C_E\|v\|_E,\\
\forall w\in F, & D_F\|w\|_F \leq \|w\|_F' \leq C_F\|w\|_F,
\end{eqnarray*}

then $$(D_F/C_E)^k\|\varcurlywedge^kA\|\leq \|\varcurlywedge^kA\|'\leq (C_F/D_E)^k\|\varcurlywedge^kA\|.$$
\end{lemma}

\begin{proof}
		
By the singular value decomposition there exists an orthonormal family $(e_i)_{i=1,\cdots, k}$ of $(E, \|\cdot\|_E)$ such that $(Ae_i)_{i}$ is an orthogonal family in $(F,\|\cdot\|_F)$ with $\|\varcurlywedge^kA\|=\|Ae_1 \cdots \wedge Ae_k\|_F=\prod_{i=1}^k\|Ae_i\|_F$. 
Similarly we let $(e'_i)_{i=1,\cdots, k}$ be the corresponding orthonormal family for the norms $\|\cdot\|'_E$ and $\|\cdot\|'_F$. Let $P$ be the change of basis matrix from $(e'_i)_i$ to $(e_i)_i$. Then the norms 
$\|e_1\wedge\cdots \wedge e_{k}\|_E'$ and $\|e'_1\wedge\cdots \wedge e'_k\|_E$  are just given by the absolute values of the determinants of $P$ and $P^{-1}$ respectively. Therefore we have
\begin{eqnarray*}
|\det(P^{-1})|&\leq &\prod_{i}\|e'_i\|_E\\[2mm]
&\leq& D_{E}^{-k},
\end{eqnarray*}
and 
\begin{eqnarray*}
\|e_1\wedge\cdots \wedge e_{k}\|_E'&=&|\det(P)|\\[2mm]
&=&1/|\det(P^{-1})|\\[2mm]
&\geq &D_{E}^{k}.
\end{eqnarray*}
We conclude that 
\begin{eqnarray*} 
\|\varcurlywedge^kA\|'&\leq &\frac{\|Ae_1\wedge \cdots\wedge Ae_k\|_F'}{\|e_1\wedge \cdots\wedge e_k\|'_E}\\[2mm]
&\leq & D_E^{-k}\prod_i\|Ae_i\|'_F \\[2mm]
&\leq & D_E^{-k}C_F^k\prod_i\|Ae_i\|_F\\[2mm]
&\leq & (C_F/D_E)^k\|\varcurlywedge^kA\|.
\end{eqnarray*}
The other inequality is obtained symmetrically.

\end{proof}

\end{document}